\documentclass[a4paper,12p]{article}
\usepackage[english]{babel}
\usepackage[utf8]{inputenc}
\usepackage[T1]{fontenc}
\usepackage{amsmath}
\usepackage{amsthm}
\usepackage{amssymb}
\usepackage{amsfonts} 
\usepackage{fullpage}
\usepackage{cleveref}
\usepackage{url}
\usepackage{multicol}
\usepackage{multirow}
\usepackage{todonotes}
\usepackage{subcaption}
\usepackage{wasysym}
\DeclareMathOperator{\ex}{ex} %extremal/Tur\'an number
\DeclareMathOperator{\bal}{bal} %balancing number
\newtheorem{theorem}{Theorem}[section]
\newtheorem{lemma}[theorem]{Lemma}
\newtheorem{definition}[theorem]{Definition}

\newtheorem{corollary}[theorem]{Corollary}
\newtheorem{example}[theorem]{Example}

\title{The evolution of unavoidable bichromatic patterns and extremal cases of balanceability}
\author{
 }
\date{}

\begin{document}
	
	\maketitle
	
	\begin{center}

\begin{multicols}{2}

Yair Caro\\[1ex]
{\small Dept. of Mathematics\\
University of Haifa-Oranim\\
Tivon 36006, Israel\\
yacaro@kvgeva.org.il}

\columnbreak

Adriana Hansberg\\[1ex]
{\small Instituto de Matem\'aticas\\
UNAM Juriquilla\\
Quer\'etaro, Mexico\\
ahansberg@im.unam.mx}\\[2ex]

\end{multicols}

Amanda Montejano\\[1ex]
{\small UMDI, Facultad de Ciencias\\
UNAM Juriquilla\\
Quer\'etaro, Mexico\\
amandamontejano@ciencias.unam.mx}\\[4ex]

\end{center}

\begin{abstract}
We study the color patterns that, for $n$ sufficiently large, are unavoidable in $2$-colorings of the edges of a complete graph $K_n$ with respect to $\min \{e(R), e(B)\}$, where $e(R)$ and $e(B)$ are the numbers of red and, respectively, blue edges. More precisely, we completely characterize which patterns are unavoidable depending on the order of magnitude of $\min \{e(R), e(B)\}$ (in terms of $n$), and show how these patterns evolve from the case without restriction in the coloring, namely that $\min \{e(R), e(B)\} \ge 0$ (given by Ramsey's theorem), to the highest possible restriction, namely that $|e(R) - e(B)| \le 1$. We also investigate the effect of forbidding certain sub-structures in each color. In particular, we show that, if the graphs induced by each of the colors are both free from an induced matching of certain size, the patterns that were unavoidable with $\Omega(n^{2-\varepsilon})$ edges in each color are already granted when $\min \{e(R), e(B)\}$ is linear. Moreover, we analyze the consequences of these results to the balancing number $\bal(n,G)$ of a graph $G$ (i.e. the minimum $k$ such that every $2$-edge coloring of $K_n$ with $\min \{e(R), e(B)\} > k$ contains a copy of $G$ with half the edges in each color), and show that, for every $\varepsilon > 0$, there are graphs $G$ with $\bal(n,G) = \Omega(n^{2-\varepsilon})$, which is the highest order of magnitude that is possible to achieve, as well as graphs where $\bal(n,G) \le c(G)$, where $c(G)$ is a constant that depends only $G$. We characterize the latter ones. We also extend that investigation of the effect of forbidding induced monochromatic matchings of given size to the multicolor setting. 
\end{abstract}
	
\section{Introduction}

\subsection{State of the art}
Ramsey's theorem \cite{Ram} (see also \cite{GrSp}) states that, for any positive integer $t$,  there is an integer $R(t)$ such that, if $n \ge R(t)$, then any $2$‐edge coloring of $K_n$ contains a monochromatic copy of $K_t$. A classical result of Erd\H{o}s and Szekeres \cite{ErSz} implies that $R(k) < 2^{2t}$ for every positive integer $t$. For a recent breakthrough see \cite{CGMS23}. On the other hand,  Turán's theorem \cite{Tur} (see also \cite{AiZie}) guarantees, for a given integer $n > t$, the existence of a subgraph $K_t$ contained in any graph $G$ of order $n$ that has at least $(1 - \frac{1}{t-1}) \frac{n^{2}}{2}$ edges. To force the existence of non-monochromatic color patterns, one needs, as a natural minimum requirement, to ensure both, a large enough $n$ as well as a minimum amount of edges in each color class.  In some settings, the effect of forbidding certain subgraphs in the coloring or graph in question has also been studied \cite{AlSh, Ill, LTTZ}. In this  line, the famous Erd\H{o}s-Hajnal conjecture \cite{ErHa} (see also \cite{Chud}) states that graphs defined by forbidden induced subgraphs have either large cliques or large independent sets. More precisely, it says that, given a graph $H$, there is a constant $\delta = \delta(H)$ such that every graph $G$ of order $n$ that is free of an induced $H$ contains a clique or an independent set of size $\Omega(n^{\delta})$. To compare it to Ramsey's theorem, this would mean, in terms of $2$-edge colorings of the complete graph $K_n$, that the existence of a monochromatic $K_t$ is forced already for $n = \Omega(t^{\frac{1}{\delta}})$ if the edges in one of the colors induces a graph free of an induced subgraph $H$.

Considering $2$-color patterns, the following is a natural question to ask:  which color patterns are unavoidable (if any) in every  $2$-edge-coloring of $K_n$ where $n$ is large enough and each color appears in a positive fraction of $E(K_n)$?  It was conjectured by Bollobás (see \cite{cutler, FoSu}) and shown by Cutler and Montágh \cite{cutler} that, for every $\epsilon \in (0, \frac{1}{2}]$ and every positive integer $t$, there is an $N_0=N_0(\epsilon,t)$ such that every $2$-edge-coloring of $K_n$ with $n\geq N_0$ which has at least $\epsilon \binom{n}{2}$ edges in each color class contains a complete graph of order $2t$ where one color forms either a clique of order $t$ or two disjoint cliques of order $t$.  The Ramsey-type parameter $N_0$ concerning this result was further studied by Fox and Sudakov who showed that $N_0 < \epsilon^{-ct}$ for some absolute constant $c$, and this is tight up to the value of $c$  \cite{FoSu}. It was shown in \cite{CHM-K_4} that, for $n \equiv 0, 1 \pmod 4$, there is a $2$-edge-coloring of $K_n$ in which each color has exactly $1/2$ the edges of $K_n$ and such that one of the colors induces either a clique or two disjoint cliques. Hence, even if we assume $\epsilon$ as large as possible, namely $\epsilon = \frac{1}{2}$, the only $2$-colored unavoidable patterns we can hope for have to be contained in these two particular colorings of $K_n$. 
It was shown in \cite{CHM_unav} that, in $2$-edge-colorings of $K_n$, the existence of a copy of $K_{2t}$ where one color forms either a clique of order $t$ or two disjoint cliques of order $t$ is forced as soon as certain a subquadratic amount of edges for both colors is surpassed. To discuss this extremal aspect of the problem precisely, we define a Turán-type parameter, see Definition \ref{def} below. 

Throughout the paper, we will use the following notation. For a graph $G$, we denote by $V(G)$ and $E(G)$ its vertex set and edge set, respectively. Given a $2$-edge-coloring of $K_n$, that  is, a mapping $f: E(K_n) \to  \{red, blue\}$ or, equivalently, a partition $E(K_n) = E(R) \cup E(B)$  where we implicitly assume that  $R$ and $B$ are the graphs induced by the red and the blue edges, respectively, we denote by $e(R)$ and $e(B)$  the number  of red and blue edges. Given a set of vertices $U\subseteq V(K_n)$, we denote by $R[U]$ (respectively, $B[U]$) the graph induced by the red edges (respectively, by the blue edges) having both end vertices in $U$.  For a given graph $H$ of order $n(H)\leq n$, we say that a $2$-edge-coloring of $K_n$ contains an \emph{induced monochromatic} copy of $H$ if there is a set of vertices $U\subseteq V(K_n)$, with $|U|=n(H)$, such that either $R[U]\cong H$ or $B[U]\cong H$. Given two graphs $G$ and $H$, we denote by $G \cup H$ the disjoint union of them, and $\overline{G}$ denotes the complement of the graph $G$, i.e. the graph defined on the same vertex set $V(G)$ and edge set $\overline{E(G)}$. Finally, for an integer $r \ge 1$, $rG$ will denote the graph consisting of $r$ vertex disjoint copies of $G$. Further, we denote by $K_{s,t}$ the \emph{complete bipartite graph} with partition sets on $s$ and $t$ vertices, and by $S_{s,t}$ the \emph{complete split graph} which consists of one set on $s$ vertices building a clique, another set of $t$ independent vertices, and between both sets all $st$ edges are present. 

When looking for a complete graph $K_q$ with certain red-blue pattern, say where the blue edges induce a graph $G_b$ and the red edges induce the graph $\overline{G_b}$, inside a complete graph $K_n$ equipped with a red-blue coloring of its edges, where $q \le n$, one can instead search for an induced blue copy of $G_b$ or an induced red copy of $\overline{G_b}$. Note that it is the same to look for a monochromatic induced copy of a graph $G$ or of a monochromatic induced copy of its complement $\overline{G}$.  
In this sense, the $2$-colored patterns mentioned above can be described as induced monochromatic members from $\{K_{t,t}, S_{t,t}\}$. We now give the following definition.

\begin{definition}\label{def}
For a family of graphs $\mathcal{F}$ not containing complete or empty graphs, 
let $\ex_2(K_n, \mathcal{F})$ be the largest integer $m \le \frac{1}{2} \binom{n}{2}$ such that, for large enough $n$, there is a $2$-edge coloring of $K_n$ where the smallest color class is of size $m$ and such that it contains no members of $\mathcal{F}$ as induced monochromatic subgraphs. If there is no such $m$, we set $\ex_2(K_n, \mathcal{F})=\infty$. 
\end{definition} 

Equivalently, $\ex_2(K_n, \mathcal{F})$ is the minimum integer $m$ such that every $2$-edge coloring of $E(K_n)$ with more that $m$ edges in each color class contains an induced monochromatic member of $\mathcal{F}$. If $\ex_2(K_n, \mathcal{F})=\infty$, this means that there we will be, for an infinite set of $n$'s, a coloring of the edges of $K_n$ having no monochromatic induced member from $\mathcal{F}$. Note that we are excluding complete and empty graphs as members of the family $\mathcal{F}$ because, by Ramsey's theorem, \emph{every} $2$-edge coloring of $E(K_n)$ contains an (induced) monochromatic complete graph $K_k$ or $\overline{K_k}$, for large enough $n = n(k)$, and clearly these graphs are the only cases where such a situation happens (just consider a monochromatic coloring). See \cite{AxGo, FoSu_indRam, Ill} for related Ramsey problems concerning induced subgraphs.

The result by Cutler and Montágh described above yields that 
\begin{equation}\label{eq:CM}
\ex_2(K_n, \mathcal{F}_t)= \mathcal{O}(n^2)
\end{equation}
where $\mathcal{F}_t = \{K_{t,t}, S_{t,t}\}$, for $t \ge 2$. Moreover, it was shown in \cite{CHM_unav}  that, for  every integer $t \ge 2$,  there exists a $\delta=\delta(t)$ such that 
\begin{equation}\label{eq:CHM}
\ex_2(K_n, \mathcal{F}_t) = \mathcal{O}(n^{2-\delta}).
\end{equation} 
Shortly after,  Gir\~ao and Narayanan \cite{GiNa} proved that,  for every $t\geq 3$,  $\delta(t)=1/t$, that is:
\begin{equation}\label{eq:GiNa}
\ex_2(K_n, \mathcal{F}_t) = \mathcal{O}(n^{2-\frac{1}{t}}),
\end{equation}
and they showed, conditioned to the veracity of the well-known K\H{o}vari-S\'os-Tur\'an conjecture claiming $\ex(n, K_{t,t})=\Omega(n^{2-1/t})$ \cite{KST}, that this is best possible up to the involved constants. Moreover, they exhibited the degeneracy of the case $t=2$ \cite{GiNa} where it turns out that
\begin{equation}\label{eq:GiNa_t=2}
\ex_2(K_n, \mathcal{F}_2) = \Theta (n).
\end{equation}

For more recent results concerning unavoidable patterns in colorings of the complete graph, we refer the reader to \cite{BHMM, BLM, GiHa, GiMu, MuTa}. 

\subsection{Our current contribution}
In Section \ref{sec:evolution}, we analyze how the unavoidable patterns evolve with respect to the order of magnitude of the number of edges that are assumed in each color (with respect to $n$, the order of the complete graph that is $2$-colored). We divide the section into three parts, distinguishing among subquadratic, linear and constant orders. In each part, we have a general result that yields the unavoidable patterns guaranteed by this order of magnitude of the number of edges in each color and we prove the sharpness. Moreover, we also analyze the effect of forbidding an induced matching in each of the colors. Similar to the effect on Ramsey's theorem that is assumed to happen in view of the Erd\H{o}s-Hajnal conjecture, it turns out that, in $2$-colorings of $K_n$ with forbidden substructures in each color, the $2$-colored unavoidable patterns emerge already with much weaker assumptions on the presence of the colors. We describe now the results with more detail.

In Section \ref{subsec:subquadr}, we generalize the result in (\ref{eq:GiNa}) establishing an asymmetric and stronger version of the result by showing that, for positive integers $s, t$ with $t \ge \max\{2, s\}$:
\begin{equation}\label{eq:main}
\ex_2(K_n, \{K_{s,T},S_{t,t}\}) = \mathcal{O}(n^{2-\frac{1}{s}}),
\end{equation} where, again, $T$ is a large number that depends on $t$ (see Theorem \ref{thm:general} for a precise statement).

We would like to emphasize here that the proof of this result is different from the proofs of the symmetric case. The proof of (\ref{eq:GiNa}) given in \cite{GiNa} is quite entangled and the case $t = 2$ has to be covered separately. On the other hand, based on the short and concise proof of (\ref{eq:CHM}) given in \cite{CHM_unav} and with the same dependent random choice argument as in \cite{GiNa},  a very short proof of (\ref{eq:GiNa}) is given again in \cite{BHMM}. However, adapting the latter proof to an asymmetric version yields a much weaker result than what we get in (\ref{eq:main}), see the discussion after the proof of Theorem \ref{thm:general}.

By means of (\ref{eq:main}), it becomes evident how unavoidable patterns evolve with respect to the granted minimum conditions on the amount of edges in each color. Let $s, t$ be positive integers with $t \ge \max\{2, s\}$, and define $\mathcal{F}_{s,t} = \{K_{s,t}, S_{t,t}\}$ (see Figure \ref{fig:patterns} for an illustration of the monochromatic induced patters determined by this family). Note that, when $s = t$, we simply set $\mathcal{F}_t$ instead of $\mathcal{F}_{t,t}$, coinciding with the definition above. Moreover, making use of the Alon-Rónyai-Szabó construction \cite{ARS}, we show that $\ex_2(K_n, \mathcal{F}_{s,t}) = \Omega(n^{2-\frac{1}{s}})$ for $s \ll t$, settling unconditional sharpness in (\ref{eq:GiNa}). Moreover, we prove that, when $s = 1$ or $s = 2$, (\ref{eq:main}) is sharp for every $t \ge s+1$ (the case $s = t = 2$ is already covered in (\ref{eq:GiNa_t=2})), see Theorem \ref{thm:st_sharp}. 

 \begin{figure}[h]
\begin{center}
	\includegraphics[scale=0.5]{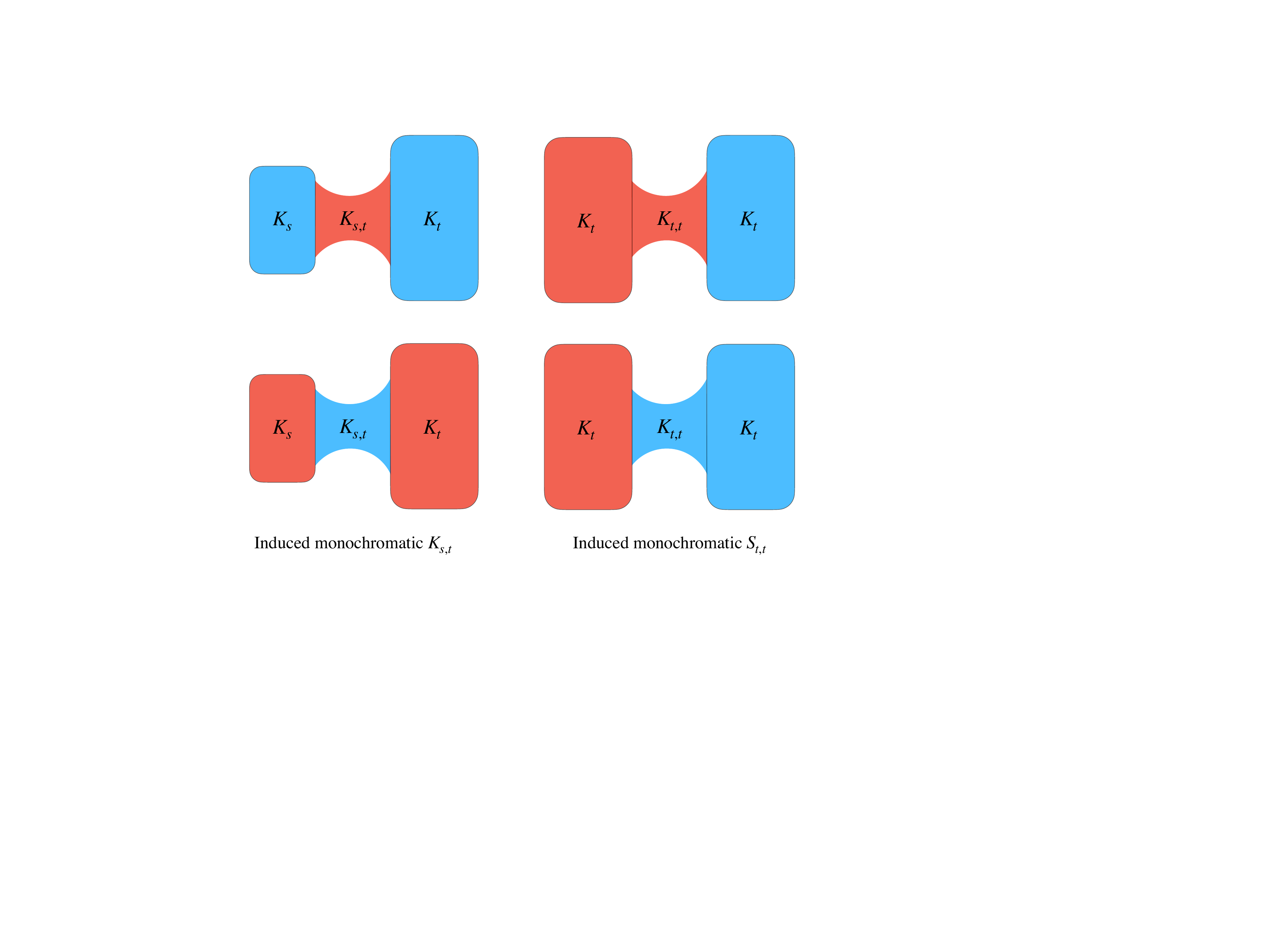}
	\caption{Colored patterns obtained by the family $\mathcal{F}_{s,t}= \{K_{s,t}, S_{t,t}\}$.}
		\label{fig:patterns}
	\end{center}
\end{figure}

We continue in Section \ref{subsec:linear} with an assumption of an $O(n)$-amount of edges in each color. Observe that Theorems \ref{thm:general} and \ref{thm:L_s,t-sharp} include the linear case, so the general results on unavoidable patterns for this setting are actually corollaries from the latter (Corollaries \ref{cor:star} and \ref{cor:star_sharp}). However, we want to make this distinction because it is very interesting to see what happens if, for some  positive integer $r$, an induced matching $rK_2$ is forbidden in each of the colors. It turns out that, in such a setting, the patterns in $\mathcal{F}_{s,t}$, where $t \ge s \ge 2$, are already unavoidable with just a linear amount of edges in each color, see Theorem \ref{thm:linear-tK_2}. In other words, we have that, with at least a linear amount of edges in each color, either an induced monochromatic $K_{s,t}$, or an induced monochromatic $S_{t,t}$, or an induced monochromatic $rK_2$ is unavoidable. Moreover, it is shown in Theorem \ref{thm:L_s,t-sharp} that this is asymptotically sharp. Setting $\mathcal{L}_{r,s,t} =\{r K_2\} \cup \mathcal{F}_{s,t}$, we thus have 
\begin{equation}
    \ex_2(K_n, \mathcal{L}_{r,s,t})= \Theta(n).
\end{equation}

Section \ref{subsec:constant} deals with the constant case. That is, assuming there is at least certain (large enough) constant amount of edges in each color, we want to know if there are some $2$-colored unavoidable patterns. Indeed, we are able to show that either a monochromatic induced star $K_{1,t}$ or a monochromatic induced matching $rK_2$ is unavoidable. This is clearly asymptotically sharp, and thus
\begin{equation}
    \ex_2(K_n, \mathcal{L}_{r,1,t})= \Theta(1).
\end{equation}
This is especially interesting because of its relation to Delta-systems and the fact that $\ex(n, \{K_{1,t}, tK_2\}) = \mathcal{O}(1)$ \cite{ABHS,ErRa}. We finish Section \ref{sec:evolution} showing that all the obtained $2$-colored patterns are unique, meaning that, given the condition on the minimum amount of edges in each color, we cannot hope for other unavoidable patterns. \\

In the context of balanceable graphs, which were defined in \cite{CHM_unav, CHM-K_4}, we investigate in Section \ref{sec:bal} what are the extremal cases of the \emph{balancing number} $\bal(n,G)$ with respect to the set of all graphs that are balanceable. To be precise, a graph $G$ is \emph{balanceable} if, for sufficiently large $n$, there is a minimum integer $k = k(n)$ such that, any $2$-edge coloring of $K_n$ with more than $k$ edges in each color contains a \emph{balanced} copy of $G$, i.e. a copy of $G$ with half (ceiling or floor) the edges in each color. If this integer $k$ exists, we set $\bal(n,G) = k$, else, namely if $G$ is not balanceable, $\bal(n,G) = \infty$. Balanceable graphs have been characterized in \cite{CHM_unav} by means of the unavoidable patterns $\mathcal{F}_t$ (see Theorem \ref{thm:char-BP}), and have been further studied in \cite{CHM-K_4, CLZ19, DHV20, DHV_Kn, EHMV_Fibo}. By (\ref{eq:GiNa}) and this characterization, it follows that a balanceble graph $G$ of order $t$ has $\bal(n,G) = \mathcal{O}(n^{2-\frac{1}{t}})$. We show here that, for every $\varepsilon > 0$, there are infinitely many graphs $G$ for which $\bal(n,G) = \Omega(n^{2-\varepsilon })$ (Theorem \ref{thm:large_bal}). On the other extreme, there are also graphs where the balancing number is just a constant $C$ depending only on the graph $G$ (not on $n$), i.e. $\bal(n,G) = C$. In Theorem \ref{thm:characterization}, we characterize the latter ones.\\

As explained above for the constant case, the induced monochromatic matching is one of the possible unavoidable patterns. However, we can also see it in the setting of forbidding it as substructure in each of the colors, which has the effect of forcing the induced monochromatic star. This is also the reason why we consider induced monochromatic matchings in the linear case, too. More precisely, we have seen that, in $2$-colorings of $K_n$ with forbidden induced monochromatic matchings of certain size, the $2$-colored unavoidable patterns $\mathcal{F}_{s,t}$ emerge already with much weaker assumptions on the presence of the colors. In view of the Erd\H{o}s-Hajnal conjecture, this resembles the effect on Ramsey's theorem that is assumed to happen when forbidding certain substructures. A natural way to continue this investigation here is to explore what is the effect of forbidding other substructures. We answer this question in Section \ref{sec:effect}. \\

Finally, we conclude in Section \ref{sec:multic} presenting a possible multicolor version of the effect of forbidding monochromatic induced $rK_2$'s.\\

\section{Evolution of the unavoidable patterns}\label{sec:evolution}
In this section we study how unavoidable colored patterns evolve with respect to the minimum requirement on the presence of edges in each color class that is asked. As it was mentioned before, provided $n$ is sufficiently large, no restrictions on the coloring leaves us only to the existence of (induced) monochromatic cliques (Ramsey's theorem); on the other hand, we already mentioned that every $2$-edge-coloring of $K_n$ with a positive fraction of  $\binom{n}{2}$ in each color class forces the existence of an induced monochromatic member from $\mathcal{F}_{t} = \{K_{t,t}, S_{t,t}\}$  \cite{cutler, FoSu}, and that these are the only possible unavoidable patterns that we can hope for, even if we assume to have $\frac{1}{2}\binom{n}{2}$ edges in each color class. Here, we study what happens in between. First we show that, for all positive integers $s$ and $t$ with $1 \le s \leq t$, a subquadratic amount  (on  $n$) of edges in each color class forces the existence of an induced monochromatic member from $\{K_{s,T}, S_{t,t}\}$ where $T$ is a large number depending on $t$ (see Theorem \ref{thm:general} for a precise statement);  in particular, this implies the existence of an induced monochromatic member of $\mathcal{F}_{s,t} = \{K_{s,t}, S_{t,t}\}$ which  provides a stronger version of Theorem 1.1 in \cite{GiNa} that deals with the case $s=t$ (see also, Theorem~2.3 in \cite{BHMM}). In order to prove Theorem \ref{thm:general}, we make use of the following version of the dependent random choice lemma (see also \cite{GiNa}).

\begin{lemma}[\cite{FoSu_DRC}]\label{lemma:drc}
For all $K,s\in\mathbb{N}$, there exists a constant $C=C(K,s)$ such that any graph with at least $Cn^{2-1/s}$ edges contains a set $S$ of $K$ vertices in which each subset $X\subseteq S$ with $s$ vertices has a common neighborhood of size at least $K$. 
\end{lemma}

\subsection{Subquadratic amount of edges in each color}\label{subsec:subquadr}

 We denote by $R(k)$ the (diagonal) $2$-color-Ramsey number (that is, the smallest integer for which every $2$-edge-coloring of $K_n$, with $n\geq R(k)$, contains a monochromatic $K_k$), and  by $BR(t)$ the bipartite Ramsey number (that is, the smallest integer for which every $2$-edge-coloring of $K_{n,n}$, with $n\geq BR(t)$, contains a monochromatic $K_{t,t}$). \\
 
\begin{theorem}\label{thm:general}
For every pair of positive integers $s$ and $t$ with $t \ge \max\{2, s\}$, there is a constant $C= C(s,t)$ such that, if $n$ is large enough,   then  every $2$-edge-coloring $E(K_n) = E(R) \cup E(B)$ with $\min \{ e(R), e(B)\} \ge C n^{2-\frac{1}{s}}$  contains an induced monochromatic member from $\{K_{s,BR(t)},S_{t,t}\}$, in particular, an induced monochromatic member of $\mathcal{F}_{s,t}$.
\end{theorem}

\begin{proof}
Given $s$ and $t$, positive integers with $t \ge \max\{2, s\}$, let $C$ be the constant provided by Lemma~\ref{lemma:drc} for $s$ and $K=R(BR(t)+1)$. Then, for every $n$ large enough and any $2$-edge-coloring of $K_n$ with $Cn^{2-1/s}$ edges in each color class, we can apply Lemma \ref{lemma:drc} to the red graph $R$ as well as to the blue graph $B$. In the fist case, we get a $K$-set of vertices, named $S_R$,  with the property that each $s$-subset $X\subseteq S_R$  has a common red-neighborhood of size at least $K$. Since $K=R(BR(t)+1)$,  there is a set $Y\subset S_R$ on $BR(t)+1$ vertices which induce a monochromatic complete graph. Now, choose any $s$-set $X\subset Y$ and take from its red neighbour, which satisfies $|N_R(X)|\geq K=R(BR(t)+1)$, a set $W$ that induce a complete monochromatic graph of order $BR(t)+1$. 

In summary, we have $Y$ and $W$, two (not necessarily disjoint) monochromatic cliques each of order $BR(t)+1$, and an $s$-subset $X\subset Y$ with $X \cap W = \emptyset$ such that every edge between $X$ and $W$ is red.

Now we analyze the cases according to the colors in $Y$ and $W$. If $Y$ and $W$ induce monochromatic cliques of different color,  then $|Y\cap W|\leq 1$ and so there are subsets $Y'\subset Y$ and $W'\subset W$ such that $Y'$ and $W'$ induce two disjoint monochromatic complete graphs with $BR(t)$ vertices each. Hence, there is a monochromatic $K_{t,t}$ between $Y'$ and $W'$, thus forming  an induced monochromatic copy of $S_{t,t}$. Now we can assume that $Y$ and $W$ have the same color. If this color is blue, the vertex set $X\cup W$ induce a red  $K_{s,BR(t)}$, and we are done. It remains to consider the case where both $Y$ and $W$ are red.

Keeping in mind that we have a red clique $W$ of order $BR(t)+1$, we proceed in the same way as above,  applying  Lemma \ref{lemma:drc} to the blue graph $B$, until we get either an induced monochromatic $S_{t,t}$, an induced blue $K_{s,BR(t)}$ or a set $Z$ of $BR(t)+1$ vertices inducing a blue clique. In the first two cases, we are done. In the third case, we have one red clique $W$ and one blue clique $Z$, each of order $BR(t)+1$. As before, this leads to a monochromatic copy of $S_{t,t}$, which completes the proof.   
\end{proof}

Observe that, if we consider the symmetric case $s = t$, we obtain the family $\{K_{t,BR(t)},S_{t,t}\}$, which is a slightly stronger version of (\ref{eq:GiNa}) (which is Theorem 1.1 in \cite{GiNa}). Moreover, if we adapt
 the elegant proofs of (\ref{eq:GiNa}) given in \cite{BHMM, CHM_unav}  to an asymmetric version, they yield a much weaker result than Theorem \ref{thm:general}: the most one obtains is the existence of a monochromatic induced member from $\{S_{s,T}, S_{T,s}, K_{s,T},S_{t,t}\}$, where $T = BR(t)$, which is not only a larger family than the one in Theorem \ref{thm:general} but it neither makes sense for $s=1$, where $S_{BR(t),1} \cong K_{BR(t)} \cup K_1$ is already covered by Ramsey's theorem.

In the statement of  Theorem \ref{thm:general}, we emphasize that  every $2$-edge-coloring of $K_n$, with $n$ large enough, and $C n^{2-\frac{1}{s}}$ edges in each color class contains, in particular, an induced monochromatic member from $\mathcal{F}_{s,t}=\{K_{s,t},S_{t,t}\}$ because this is a family that will be relevant in the in subsequent sections. 

Next, we analyze the sharpness of Theorem \ref{thm:general}. For this purpose,  we recall some classical parameters and known results in Extremal graph theory. The Zarankievicz number ${\rm z}(m,n;s,t)$ \cite{Zar51}, is  the maximum number of edges in a $K_{s,t}$-free bipartite graph with bipartition $X \cup Y$ such that $|X| = m$ and $|Y| = n$. We write ${\rm z}(n;t)$ instead of ${\rm z}(n,n;t,t)$. K\H{o}vari, S\'os and Tur\'an \cite{KST} showed that 
\begin{equation}\label{eq:mnst}
{\rm z}(m,n;s,t) < (s-1)^{\frac{1}{t}}(n-t+1)m^{1 - \frac{1}{t}}+(t-1)m,
\end{equation} 

and this  yields the following bound on the extremal number ${\rm ex}(n,K_{s,t})$, when $s\leq t$,  
\begin{equation}\label{eq:exst}
  {\rm ex}(n,K_{s,t}) \leq \frac{1}{2} \left((t-1)^{\frac{1}{s}}n^{2 - \frac{1}{s}}+(s-1)n \right).
\end{equation}
The K\H{o}vari-Sós-Turán conjecture states that the bound in (\ref{eq:exst})  gives the correct order of magnitude for every  $s,t \ge 2$ \cite{KST}. In other words,  it is believed that there are $K_{s,t}$-free graphs of order $n$ with $\Omega(n^{2 - \frac{1}{s}})$ edges for all integers $2\leq s\leq t$. This conjecture has been confirmed for $s=t = 2$ \cite{ERS66, Fur96}, $s = t = 3$ \cite{ARS, Bro66} and when $t \ge (s-1)! +1$ and $s \ge 2$ \cite{ARS, KRS} (this includes the previous cases, too); moreover, in all these cases, the graphs exposed to confirm the conjecture are bipartite graphs. Making use of the latter, we  will prove the following.

\begin{theorem}\label{thm:st_sharp}
For every $n$ sufficiently large and positive integers $s,t$ such that $t \ge \max\{3, (s-1)! + 1\}$, we have $\ex_2(K_n,\mathcal{F}_{s,t})= \Theta(n^{2-\frac{1}{s}})$.
\end{theorem}
\begin{proof}
For the case that $s=1$ and $t \ge 3$, it is enough to exhibit a $2$-edge-coloring of $K_n$ with a linear number of edges in each color class but without monochromatic induced stars $K_{1,t}$. Let $E(K_n)= E(R)\cup E(B)$ where $B \cong \lfloor \frac{n}{t-1} \rfloor K_{t-1}$. Then, if $n$ is large enough, we have  $e(R)\geq e(B)= \lfloor \frac{n}{t-1} \rfloor \binom{t-1}{2}$, and clearly this coloring contains no monochromatic induced $K_{1,t}$. Similarly, for 
 $s \ge 2$ and $t \ge \max\{3, (s-1)! + 1\}$, the upper bound is given by Theorem \ref{thm:general}. For the lower bound, let $n$ be sufficiently large and consider the $2$-edge-coloring $E(K_n)= E(R) \cup E(B)$ where $B$ is the $K_{s,t}$-free bipartite graph on $Cn^{2 - \frac{1}{s}}$ edges constructed in \cite{ARS, KRS} and mentioned previous to this theorem. Observe that, since $B$ is bipartite and $t\geq 3$, the coloring cannot contain an induced monochromatic copy of a member from $\mathcal{F}_{s,t}$.
\end{proof}

Notice that the small cases not covered by this theorem (i.e. where $s \le 2 = t$) are those that are precisely degenerate: it is easily seen that $\ex_2(K_n, \mathcal{F}_{1,2}) = 0$ (to find an induced monochromatic copy of a member of $\mathcal{F}_{1,2}$ is equivalent to find a copy of a non-monochromatic triangle or a monochromatic induced $K_4$ minus an edge, but a non-monochromatic triangle appears as soon as there is at least one edge in each color), while $\ex_2(K_n, \mathcal{F}_{2,2}) = \Theta(n)$, as we already know by (\ref{eq:GiNa_t=2}).

\subsection{Linear amount of edges in each color}\label{subsec:linear}

The following corollary is immediate from Theorem \ref{thm:general} corresponding to the particular case $s=1$; it establishes that a linear  amount (on $n$) of edges in each color class forces, in particular,  the existence of an induced monochromatic star $K_{1,t}$.\\

\begin{corollary}\label{cor:star}
For every integer  $t \ge 2$, there is a constant $C= C(t)$ such that, if $n$ is large enough, then  every $2$-edge-coloring $E(K_n) = E(R) \cup E(B)$ with $\min \{ e(R), e(B)\} \ge C n$  contains an induced monochromatic member of $\{K_{1,BR(t)},S_{t,t}\}$, in particular, an induced monochromatic star $K_{1,t}$. 
\end{corollary}

By Theorem \ref{thm:st_sharp}, the above bound is sharp for any $t \ge 3$. In particular, we have the following.

\begin{corollary}\label{cor:star_sharp}
For every $n$ sufficiently large and any integer $t \ge 3$, $\ex_2(K_n, K_{1,t})= \Theta(n)$.
\end{corollary}

\noindent
Again, the case $t = 2$ falls out from this statement as $\ex_2(K_n, K_{1,2})= 0$, which is not difficult to see.\\

In the next theorem, we analyze the consequences of forbidding monochromatic induced matchings when an $\Omega(n)$-amount of edges in each color is guaranteed. It turns out that, in such colorings, just a linear amount of edges in each color is sufficient to force the existence of the induced monochromatic patterns from $\mathcal{F}_{s,t}$.

\begin{theorem}\label{thm:linear-tK_2}
Let $r,s,t$ be positive integers  such that $r \ge 2$, and $t \ge s \ge 2$. Then there is a constant $C= C(r,s,t)$ such that, if $n$ is large enough, then every coloring $E(K_n) = E(R) \cup E(B)$ with $\min \{ e(R), e(B)\} \ge C n$ without a monochromatic induced $rK_2$, contains an induced monochromatic member from $\mathcal{F}_{s,t}$.
\end{theorem}

Before proving this, we need the following extension of Ramsey's theorem.  Recall that, we denote by $R(s)$ the classical (diagonal) $2$-color Ramsey number, that is, the minimum integer $R(s)$ such that, whenever $n\geq R(s)$, any $2$-edge-coloring  of $K_n$ contains a monochromatic copy of $K_s$.
\begin{lemma}\label{lem:Ramsey-matrix}
Let $\ell$ and $m$ be positive integers. Then there is an integer $Q = Q(\ell,m)$, such that, if $q \ge Q$, the following holds.  If $K_n$ is a complete graph of order $n = mq$ on vertex set $\{v_{i,j} \;|\; i \in [m], j \in [q]\}$ whose edges are red-blue colored, then there is a set $L \subseteq [q]$ with $|L| = \ell$ such that, for each $i \in [m]$, the set $\{v_{i,j} \;|\; j \in L\}$ induces a monochromatic $K_{\ell}$.
\end{lemma}
\begin{proof}
Let $R^0(\ell)=\ell$, $R^1(\ell)=R(\ell)$ and, for $i\geq 1$,  let $R^{i}(\ell)=R(R^{i-1}(\ell))$. Set $Q=R^m(\ell)$. In the first step, consider the $2$-edge-colored complete graph induced by the vertices  $\{v_{1,j} \;|\;  j \in [q]\}$ where, by definition, we can find a set $L_1 \subset [q]$ with $|L_1|=R^{m-1}(\ell)$ such that $\{v_{1,j} \;|\;  j \in L_1\}$ induces a monochromatic complete graph. In the $i$th step, consider the $2$-edge-colored complete graph induced by the vertices  $\{v_{i,j} \;|\;  j \in L_{(i-1)}\}$ where, by definition, we can find a set $L_i \subset L_{i-1}$ with $|L_i|=R^{m-i}(\ell)$ such that $\{v_{i,j} \;|\;  j \in L_i \}$ induces a monochromatic complete graph. At the $m$-th step, we get a set $\{v_{m,j} \;|\;  j \in L_m \}$ inducing a monochromatic complete graph for certain $L_m \subset L_{m-1}$ with $|L_m|=R^{0}(\ell)=\ell$. Set $L=L_m$ to conclude the proof.
\end{proof}

\begin{proof}[Proof of Theorem \ref{thm:linear-tK_2}]
Let $\ell, q$ be integers with $\ell^2 \ge \max\{ 2 z(\ell,t), 4r  z(\ell,\ell;s,t) \}$, $\ell \ge t$, and $q = Q(\ell,2)$, $n \ge R(\ell) + 4q$, and let $C $ be large enough such that $C n \ge \ex(n, qK_2)$, which is possible because $\ex(n, q K_2) = \Theta(n)$. 

%Assuming we have a red-blue coloring $E(K_n) = E(R) \cup E(B)$ such that $\min\{e(R), e(B)\} \ge C n$ and such that it has no induced monochromatic $rK_2$, we can follow the steps of the proof of Theorem \ref{thm:forbidding_tH} to finally obtain an induced monochromatic $K_{s,t}$ or $S_{t,t}$. This implies that, for $n$ large enough, every $2$-edge coloring of $K_n$ with at least $C n$ edges in each color class necessarily contains a monochromatic induced $tK_2$ or a monochromatic induced member from $\mathcal{F}_{s,t}$.

%Let $m = n(H)$,  and let $\ell, q$ be integers with $\ell^2 \ge \max\{ 2 z(\ell,t), 4r \binom{m}{2} z(\ell,\ell;s,t) \}$, $\ell \ge t$, and $q \ge Q(\ell,m)$.  Let $n \ge R(\ell) + 2qm$ and let $C$ be large enough such that $C \ex(n,H) \ge \ex(n,qH)$, which exists by what we discussed before starting the proof.  

Suppose we have a complete graph $K_n$ provided with a red-blue coloring $E(K_n) = E(R) \cup E(B)$ such that $\min \{ e(R), e(B)\} \ge C  n$ and such that it has no monochromatic induced $rK_2$ (this can happen for $r \ge 2$).  Then there is both, a red and a blue copy of $qK_2$.  Consider the set $S$ of vertices that do not belong to any of these copies.  Since $|S| \ge R(\ell)$, there is a monochromatic $\ell$-clique $S'$ disjoint from the red and the blue copies of $qK_2$. Without loss of generality, we can assume that $S'$ is a red clique. Now consider the blue copy of $qK_2$ and say it has edges $v_{1,j}v_{2,j}$,  $j \in [q]$.  Since $q \ge Q(\ell,2)$, there  is a set $L \subseteq [q]$ with $|L| = \ell$ such that, for each $i \in \{1,2\}$, the set $S_i = \{v_{i,j} \;|\; j \in L\}$ induces a monochromatic $K_{\ell}$.  If, for $i \in \{1,2\}$, $S_i$ induces a blue $K_{\ell}$, then we can consider the most represented color in $E(S',S_i)$, which has at least $\frac{{\ell}^2}{2} > z(\ell,\ell;t)$ edges.  Hence, it can be easily seen that  there is a monochromatic $K_{t,t}$ with one partition set inside $S$ and the other inside $S_i$. In either case, if this $K_{t,t}$ is red or blue, we obtain an induced monochromatic $S_{t,t}$. Thus, we may assume that both cliques $S_1$ and $S_2$ are red.
Define a graph $F$ with vertex set $V(F) = \{w_j \;|\; j \in L \}$ and edge set 
\[E(F) = \left\{ w_{j_1}w_{j_2} \; |\; \{v_{1,j_1}v_{2,j_2}, v_{1,j_2}v_{2,j_1} \right\} \cap E(B) \neq \emptyset\}.\]  
If $\alpha(F) \ge r$,  this  would imply,  by the construction of $F$, that there is a blue induced $rK_2$, which is not possible. Hence,  we can assume that $\alpha(F) < r$, and by Tur\'an's bound on the independence number \cite{Tur} we have
\[r > \alpha(F) \ge \frac{\ell}{{\rm d}(F)+1} = \frac{{\ell}^2}{2e(F)+\ell},\]
implying that $e(F) > \frac{{\ell}^2}{2r} - \frac{\ell}{2} \ge \frac{{\ell}^2}{4r}$.  By construction of the graph $F$, every edge $e \in E(F)$ corresponds to at least one blue edge $e' \in \{v_{1,j_1}v_{2,j_2}, v_{1,j_2}v_{2,j_1}\}$ for some pair $j_1, j_2 \in L$.  Moreover, since inside the sets $S_1$ and $S_2$ there are only red edges, we can derive
\[e_B(S_{1}, S_{2}) = \sum_{j_1, j_2 \in L} e_B(\{v_{1,j_1},v_{2,j_1}\}, \{v_{1,j_2},v_{2,j_2}\}) \ge e(F) > \frac{{\ell}^2}{4r} \ge z(\ell,\ell;s,t).\]
Hence, there is a blue $K_{s,t}$ with one partition set in $S_1$ and the other in $S_2$, i.e. an induced blue $K_{s,t}$.
\end{proof}

Regarding the sharpness of Theorems \ref{thm:linear-tK_2}, we have the following.

\begin{theorem}\label{thm:L_s,t-sharp}
Let $r, s, t$ and $n$ be  positive integers, where $t \ge s \ge 2$, and let $\mathcal{L}_{r,s,t} =  \{r K_2\} \cup \mathcal{F}_{s,t}$. Then $\ex_2(K_n, \mathcal{L}_{r,s,t}) = \Theta(n)$.
\end{theorem}
\begin{proof}
The upper bound is given by Theorem \ref{thm:linear-tK_2}. To see the lower bound, consider the following coloring of the edges of $K_n$.  Let $V(K_n)=V_1\cup  ... \cup V_{w-1}\cup W$ be a partition of $V(K_n)$ with $|V_i|=t-1$, for $1\leq i \leq w-1$, where $w = \min\{r,s\}$. Color red all edges having both end vertices in $W$ and all edges with one end vertex in $V_i$ and the other in $V_j$, for every $i\neq j$; all other edges are colored blue (see Figure \ref{fig:lineal_sharp}). It is not hard to check that such a coloring contains no induced monochromatic copy of $K_{s,t}$, $S_{t,t}$ or $rK_2$.  Moreover, this coloring has $(w-1)(t-1) (n - (w-1)(t-1)) + (w-1) \binom{t-1}{2} = \Theta(n)$ blue edges and at least as many red edges for $n$ large enough.
\end{proof}

\begin{figure}[h]
\begin{center}
	\includegraphics[scale=0.5]{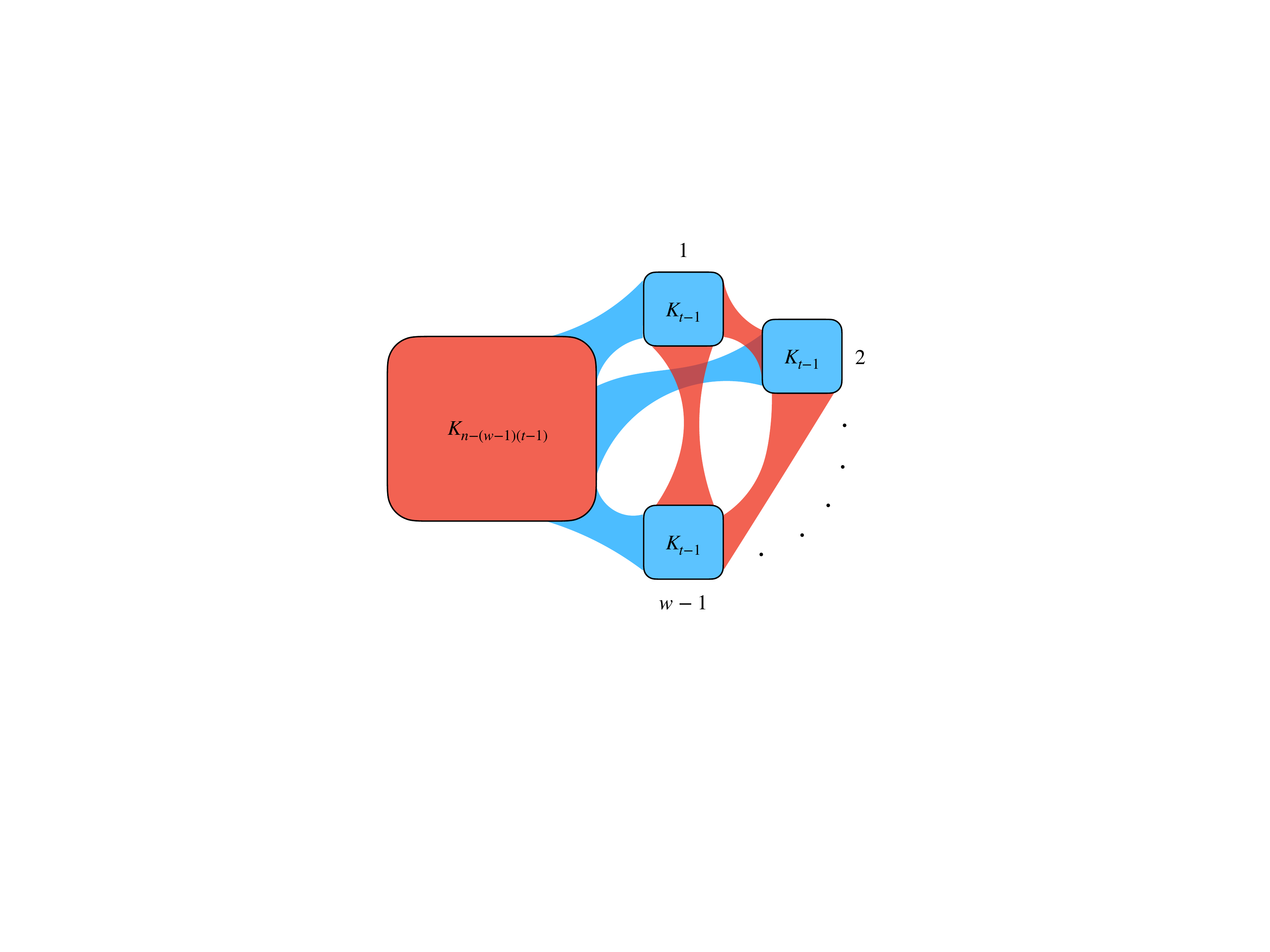}
	\caption{A $2$-edge coloring of $K_n$ with a linear amount of blue edges, but without induced monochromatic copies of  $K_{s,t}$, $S_{t,t}$, or $r K_2$, where $t \ge s \ge 2$, and $w = \min\{r,s\}$.}
		\label{fig:lineal_sharp}
	\end{center}
\end{figure}

\subsection{Constant amount of edges in each color}\label{subsec:constant}

We will finish this section studying the existence of  color patterns in $2$-edge-colorings of the complete graph under minimal restrictions.  More precisely, we give a non-trivial  answer to the following question: Does there exist a family of graphs of certain fixed order such that, for $n$ sufficiently large, the existence of an induced monochromatic  copy of one of them is forced in every $2$-edge-coloring of $K_n$ with at least a constant (not depending on $n$) number of edges in each color class?

We will make use of the \emph{strong chromatic index} $\chi'_s(G)$ of a graph $G$, which is the minimum number of colors that is required to color the edges of $G$ such that the vertices belonging to the edges of each color class induce a matching in $G$. It is well-known that a straightforward greedy argument gives 
\begin{equation}\label{eq:greedyX'_s}
\chi'_s(G) \le 2\Delta^2 - 2\Delta + 1,
\end{equation}
where $\Delta$ is the maximum degree of $G$. Erd\H{o}s and Ne\v{s}et\v{r}il conjectured that $\chi'_s(G) \le \frac{5}{4}\Delta^2$ (see \cite{FGST89}). In \cite{MoRe97}, it was shown that  $\chi'_s(G) \le 1.998 \Delta^2$ provided $\Delta$ is sufficiently large. The best current upper bound of this kind is $\chi'_s(G) \le 1.93 \Delta^2$ \cite{BrJo18}, where again $\Delta$ is assumed to be large. However, for our purposes in the next theorem the greedy bound (\ref{eq:greedyX'_s}) suffices. See also \cite{CFKP20} for related results on induced matchings.\\

\begin{theorem}\label{thm:constant}
For integers  $r, t \ge 2$, there is a constant $C = C(r,t)$ such that, for $n$ sufficiently large, every  $2$-edge-coloring of  $K_n$ with at least $C$ edges in each color class and without a monochromatic induced $rK_2$ contains an induced monochromatic $K_{1,t}$.
\end{theorem}
	
\begin{proof}
Set  $T=R(2t)$ and take $n$ large enough in order to verify the condition $\min \{e(B), e(R)\} \ge r(2T^2 -6T+5)$. That is, we need  $\binom{n}{2} \ge 2 r(2T^2 -6T+5)$, which holds if 
\begin{equation}\label{eq:hypothesis}
2n\ge 1+ \sqrt{1 + 16r(2T^2 -6T+5)}.
\end{equation} 
Since the right hand of inequality (\ref{eq:hypothesis}) is strictly grater than $2T = 2R(2t)$, there must be a monochromatic $K_{2t}$. We assume, without loss of generality, that there is a set of vertices $U$ such that $R[U] \cong K_{2t}$. Consider now the blue graph $B$ and suppose that $\Delta(B) < T$. Let $M$ be a maximum induced matching in $B$. Taking the largest chromatic class in a minimum strong edge-coloring of $B$, it is a simple fact to note that $|M| \ge \frac{e(B)}{\chi'_s(B)}$. Since the function $f(x) = 2x^2-2x+1$ is monotonically increasing for $x \ge 1$, and $\Delta(B) \le T-1$, it follows with (\ref{eq:greedyX'_s}) that
\[\chi'_s(B) \le 2\Delta(B)^2 - 2\Delta(B) + 1 \le 2 (T-1)^2 - 2(T-1) +1 = 2T^2 - 6T + 5.\] 
Hence, since $e(B) \ge r (2T^2 -6T+5)$, we obtain that
	\[|M| \ge \frac{e(B)}{\chi'_s(B)} \ge r,\]
implying that there is an induced blue $r K_2$, and we have finished. Thus, we may assume that $\Delta(B) \ge T$. Let $x$ be a vertex of maximum degree in $B$ and consider its set of neighbors $W = N_B(x)$. Since $|W| \ge T = R(2t)$, there is a monochromatic clique $W' \subset W$ of size $2t$. If $W'$ is a red clique,  then $t$ vertices from $W'$ together with $x$ provide an induced blue $K_{1,t}$, and we are done. Hence, we may assume that $W'$ is a blue clique. If $|N_R(u) \cap W'| < t$ for all vertices $u \in U$, and $|N_B(w) \cap U| < t$ for all vertices $w \in W'$, it follows that $e_R(U,W') < 2t^2$, as well as $e_B(W',U) < 2t^2$, which give the contradiction 
\[4t^2 = e(U,W') = e_R(U,W') + e_B(W',U) < 4 t^2,\]
Thus, there has to be a vertex $u \in U$ such that $|N_R(u) \cap W'| \ge t$ or a vertex $w \in W'$ such that $|N_B(u) \cap U| \ge t$. In the first case, $u$ together with $t$ vertices from  $N_R(u) \cap W'$ give an induced red $K_{1,t}$. 
In the second, $w$ together with $t$ vertices from $N_B(u) \cap U$ give an induced blue $K_{1,t}$. 
Hence, the statement follows with $C(r,t) = r(2T^2 -6T+5) = r(2(R(2t))^2 -6R(2t)+5)$.
	\end{proof}

As at least a constant number of edges is needed to have some induced monochromatic copy of any graph, Theorem \ref{thm:constant} yields the following corollary.

\begin{corollary}\label{coro:L_1,t-sharp}
Let $r, t$ and $n$ be  positive integers, where $t \ge 2$, and let $\mathcal{L}_{r,1,t} =  \{r K_2\} \cup \mathcal{F}_{1,t}$. Then $ \ex_2(K_n, \mathcal{L}_{r,1,t}) = \Theta(1)$.
\end{corollary}

\subsection{Uniqueness of the patterns}\label{subsec:uniqueness}

As mentioned in the introduction, it was shown in \cite{CHM-K_4} that, for $n \equiv 0,1 \pmod 4$, there is a $2$-edge-coloring of $K_n$ in which one of the colors induces a $K_{p,n-p}$ or a $S_{q,n-q}$, where $p$ or $q$ can be chosen in such a way that each color has exactly $1/2$ the edges of $K_n$. Hence, we cannot hope for the unavoidability of $2$-colored patterns different from those contained in these two colorings, which are precisely those that have been shown to be unavoidable and that are given by means of an induced monochromatic bipartite graph or an induced monochromatic complete split graph. Looking deeper into Theorem \ref{thm:general}, one can see that, with an $\Omega(n^{2-\frac{1}{s}})$-amount of edges in each color, we can only guarantee the $2$-colored patterns given by an induced monochromatic member of $\{K_{s,T},S_{t,t}\}$ for $t$ and $T$ of our choice, and nothing else. This is because the unavoidable patterns have to be contained in one of the two balanced colorings described above, on the one side, and on the other, because Theorem \ref{thm:st_sharp} shows that the existence of the next possible such unavoidable pattern that contains a monochromatic $K_{s,T}$ or a monochromatic $S_{t,t}$, which would be $\{K_{s+1,T},S_{t,t}\}$, is only granted if we have an $\Omega(n^{2-\frac{1}{s+1}})$-amount of edges in each color. Thus, Theorem \ref{thm:general} and Corollary \ref{cor:star} are best possible in this sense, too. 

Similarly happens with Theorem \ref{thm:constant}. Consider the following two $2$-edge-colorings of $K_n$: the first one where one color induces a $K_{1,n-1}$, and the second where one color induces a $\lfloor \frac{n}{2} \rfloor K_2$. Clearly we have in both cases  a linear amount of edges in each color class,  and the only possible induced monochromatic subgraphs different from cliques (and without isolated vertices) are $K_{1,t}$ and  $rK_2$, where $t \le n-1$, and $r \le \frac{n}{2}$. Therefore, the $2$-colored patterns represented by $\{ K_{1,t}, rK_2\}$  are essentially the only ones that can be  guaranteed having an $\Omega(1)$-amount of edges in each color. %Notice that, again, we have here a flavour of Delta-systems.\\

By considering Theorems \ref{thm:linear-tK_2} and \ref{thm:L_s,t-sharp}, an analogous anaysis can be made in the setting where monochromatic induced matchings are forbidden.\\

\section{Extremal cases of balanceability}\label{sec:bal}
	
For a  graph $G$ with $e(G)$ edges, we say that a $2$-edge-coloring $f$ of $K_n$ contains a \emph{balanced} copy of $G$, if there is a copy of $G$ in $K_n$ such that $f$ assigns one of the colors, red or blue, to exactly $\lfloor e(G)/2\rfloor$ edges. In case $e(G) \equiv 0$ (mod $2$), a balanced copy of $G$ has  precisely $e(G)/2$ edges in each color class. Given a positive  integer $n$ and a graph $G$, we define the \emph{balancing number} ${\rm bal}(n,G)$ by means of ${\rm bal}(n,G) = \ex_2(K_n, \mathcal{F}(G))$, where $\mathcal{F}(G) = \{H \subseteq G \;|\;  \lfloor e(G) /2 \rfloor \le e(H) \le \lceil e(G) /2 \rceil\}$.

 If ${\rm bal}(n,G)$ exists for every sufficiently large $n$, we say that $G$ is \emph{balanceable}. 

The following characterization of balanceable graphs is given in \cite{CHM_unav}.

\begin{theorem}[\cite{CHM_unav}]\label{thm:char-BP}
A graph $G$ is balanceable if and only if $G$ has both a partition $V(G)=X \cup Y$ and a set of vertices $W\subseteq V(G)$ such that
\[e(X,Y), e(W) \in \left\{\left\lfloor \frac{e(G)}{2} \right\rfloor, \left\lceil \frac{e(G)}{2} \right\rceil\right\}.\] 
Moreover, if $G$ is balanceable, ${\rm bal}(n,G) = \mathcal{O}(n^{2-\frac{1}{m}})$, where $m = m(G)$ depends only on $G$.
\end{theorem}

Observe that Theorem \ref{thm:char-BP} states that ${\rm bal}(n,G)$ is subquadratic with respect to $n$. By (\ref{eq:GiNa}), it follows more precisely (by the same arguments that are used for the proof Theorem \ref{thm:char-BP}) that ${\rm bal}(n,G) = \mathcal{O}(n^{2-\frac{1}{t}})$, where $t$ is the order of $G$.

\subsection{Graphs with large balancing number}

 By means of Theorem \ref{thm:char-BP} and the construction given in \cite{ARS, KRS} mentioned previously in this paper (see discussion before Theorem \ref{thm:st_sharp}), we can prove that balanceable graphs with large edge number (close to quadratic) have balancing number $\Omega(n^{2-\frac{1}{s}})$.

\begin{theorem}\label{thm:large_bal}
Let $G_t$ be a balanceable graph of order $n(t)$ and edge number $e(t)$, where $t \in \mathbb{N}$ and $n(x)$ and $e(x)$ are both monotonically increasing functions with $x$, $x >0$.  If $n(t)^{2-\frac{1}{s}} = o(e(t))$ for some integer $s \ge 2$, then $\bal(n,G_t) = \Omega(n^{2-\frac{1}{s}})$.
\end{theorem}
\begin{proof}
Let $r$ and $s$ be integers with $r \ge (s-1)!+1$. By (\ref{eq:exst}) and the result in \cite{ARS, KRS} mentioned above, there are constants $c = c(r,s)$ and $C = C(r,s)$ such that $c (n(t))^{2-\frac{1}{s}} \le \ex(n(t),K_{r,s}) \le C (n(t))^{2-\frac{1}{s}}$, for large enough $t$. Also, we know that, since $n(t)$ and $e(t)$ are both monotonically increasing and $n(t)^{2-\frac{1}{s}} = o(e(t))$, we can select $t$ large enough to verify $\left\lfloor \frac{e(t)}{2} \right\rfloor > C (n(t))^{2-\frac{1}{s}}$. We take now $n \ge n(t)$ such that $\bal(n,G_t)$ exists. Let $G$ be a $K_{r,s}$-free graph on $n$ vertices and $c n^{2-\frac{1}{s}}$ edges. Define a $2$-edge coloring $E(K_n) = E(R) \cup E(B)$ of $K_n$ such that $B \cong G$. Assume to the contrary that there is a copy $H \subset K_n$ of $G_t$ which is balanced. Then
\[ e_B(H) \ge \left\lfloor \frac{e(t)}{2} \right\rfloor > C (n(t))^{2-\frac{1}{s}},\]
implying that the graph induced by the blue edges in $H$ contains a $K_{r,s}$, a contradiction to the assumption that $B$ is $K_{r,s}$-free. Thus the given edge coloring of $K_n$ does not contain a balanced $G_t$ and we can conclude that $\bal(n, G_t) \ge e(G) \ge C n^{2-\frac{1}{s}}$. Hence, it follows that ${\rm bal}(n,G_t) = \Omega(n^{2-\frac{1}{s}})$.
\end{proof}
Theorem \ref{thm:large_bal} shows that there are graphs with very large balancing number.  We will give some examples for which this theorem can be applied.

\begin{example}\label{ex:big_bal}
The following graphs are balanceable and have large balancing number for large $t$.
\begin{enumerate}
    \item[(a)] The complete bipartite graph $K_{t,t}$.
    \item[(b)] The graph $2K_t$, where $t$ is the sum of two square natural numbers.
    \item[(c)] The graph $H_t$ with $V(H_t) = A \cup B$ such that $A = \{v_1, v_2, \ldots, v_t\}$ and $B = \{v_{t+1}, v_{t+2}, \ldots, v_{2t}\}$, where $B$ is a clique, $A$ is an independent set and adjacencies between $A$ and $B$ are given by $v_i v_{t+j}\in E(H_t)$ if and only if $j \le i$, where $1\le i, j\le t$.
    \item[(d)] The graph $E_t$ with $V(E_t) = A \cup B$ such that $A = \{v_1, v_2, \ldots, v_t\}$ and $B = \{v_{t+1}, v_{t+2}, \ldots, v_{2t}\}$, where $A$ and $B$ are independent sets and adjacencies between $A$ and $B$ are given by $v_i v_{t+j}\in E(H_t)$ if and only if $j \le i$, where $1\le i, j\le t$.
\end{enumerate}
\end{example}

\begin{proof}
Observe that, with $e(K_{t,t}) = t^2$, $e(2K_t) = t(t-1)$, $e(H_t) = t^2$, $e(E_t) = \frac{1}{2}t(t+1)$, the number of edges in each of the examples is quadratic in $t$. Moreover, $n(K_{t,t}) = n(2K_t) = n(H_t) = n(E_t) = 2t$ is linear in $t$. Hence, in all cases $n(t) ^{2-\frac{1}{s}} = o(e(t))$ for any $s \ge 2$. Since all functions are monotonically increasing with $t$, by Theorem \ref{thm:large_bal} we have $\bal(n,G_t) = \Omega(n^{2-\frac{1}{s}})$ for some $s \ge 2$ and any $G_t \in \{K_{t,t}, 2K_t, H_t, E_t\}$ whenever the four graphs are balanceable. Indeed, $2K_t$ is balanceable if and only if $t$ is the sum of two square natural numbers \cite{DHV_Kn} (see \cite{Ven23}), $H_t$ and $E_t$ are always balanceable because they are global amoebas \cite{CHM_amo, EHMV_Fibo} and global amoebas are balanceable as shown in \cite{CHM_unav}. It remains to show that $K_{t,t}$ is balanceable, as we prove below. 

Let $H \cong K_{t,t}$ be a complete bipartite graph with bipartition $V(H) = A \cup B$. Let $A = A_1 \cup A_2$, and $B = B_1 \cup B_2$ be partitions of $A$ and $B$ such that $|A_1|  = |B_1| = \lfloor \frac{t}{2} \rfloor$. Then, for the sets $X = A_1 \cup B_1$ and $Y = A_2 \cup B_2$, we have 
\[e(X,Y) = |A_1||B_2| + |A_2||B_1| = 2 \left\lfloor \frac{t}{2} \right\rfloor \left\lceil \frac{t}{2} \right\rceil = \left\lfloor \frac{t^2}{2} \right\rfloor = \left\lfloor \frac{e(K_{t,t})}{2} \right\rfloor.\]
On the other hand, let $A' \subseteq A$ such that $|A'| = t$, if $t$ is even, and $|A'| = t-1$, if $t$ is odd, and let $B' \subset B$ with $|B'| = \lceil \frac{t}{2}\rceil$. Define $W = A' \cup B'$. Then
\[e(W) = |A'||B'| = \left\lfloor \frac{t^2}{2} \right\rfloor = \left\lfloor \frac{e(K_{t,t})}{2} \right\rfloor.\]
Hence, $H \cong K_{t,t}$ fulfills the conditions of Theorem \ref{thm:char-BP} and is therefore balanceable.
\end{proof}

\subsection{Graphs with constant balancing number}

In contrast to complete bipartite graphs and the other examples given in Example \ref{ex:big_bal} that have a very large balancing number, there are other graphs where the balancing number has a much smaller order of magnitude. In \cite{CHM_unav}, it is shown that the balancing number grows linearly with $n$ when $G$ is a path or a star, and the upper bound $(k-1)n$ holds in general for any tree on $k$ edges. In this section, we will show that there is an infinite family of graphs with constant balancing number, that is, there are graphs $G$ for which  ${\rm bal}(n,G) = c(G)$ for all sufficiently large $n$, where $c(G)$ is a constant depending only on $G$. Moreover, we will characterize their structure. We will make use of the \emph{matching number} $\beta(G)$ of a graph $G$, that is, the maximum cardinality of a set of pairwise independent (or non-incident) edges in $G$. 
For a positive even integer $k$, define $\mathcal{C}_k$ to be the family of graphs with 
\begin{itemize}
\item[(i)] $e(G)=k$, 
\item[(ii)] $\beta(G) \ge \frac{k}{2}$, and
\item[(iii)] a vertex $x$ with $\deg_G(x)=\frac{k}{2}$.
\end{itemize}
Before continuing, we shall note that the graphs contained in $\mathcal{C}_k$ have a rather special structure.
\begin{lemma}\label{lem:structure}
Let $G\in \mathcal{C}_k$. Then the following properties hold.
\begin{itemize}
\item[(a)] $\beta(G) \le \frac{k}{2}+1$;
\item[(b)] There is an edge $e \in E(G)$ such that $G - e$ can be decomposed into a star $K_{1, \frac{k}{2}}$ and the graph $(\frac{k}{2}-1)K_2$.
\item[(c)] For every $v\in V(G)$, either $\deg_G(v) \le 3$ or $\deg_G(v)=  \frac{k}{2}$;
\item[(d)] Let $x$ be a vertex of degree $\frac{k}{2}$ in $G$. Then there is a maximum matching containing an edge incident with $x$.
\end{itemize}
\end{lemma} 
\begin{proof}
(a) Since $G$ has $k$ edges and  a vertex $x$ of degree $\frac{k}{2}$, at most one of the edges incident with $x$ belongs to a maximum matching. Hence, $\beta(G) \le e(G) - \left( \frac{k}{2}-1 \right) = \frac{k}{2}+1$. \\
(b) The property $\beta (G)\geq \frac{k}{2}$ yields that, among the $\frac{k}{2}$ edges of $G$ not incident with $x$, at least $\frac{k}{2}-1$ belong to a matching. Hence, we can conclude that there is an edge $e \in E(G)$ such that $G - e$ can be decomposed into a star $K_{1, \frac{k}{2}}$ and the graph $(\frac{k}{2}-1)K_2$.\\
(c) By the decomposition mentioned in (b), it is clear that $G-e$ contains the vertex $x$ with degree $\frac{k}{2}$ and all remaining vertices have degree at most $2$. Hence, in $G$, we have that either $\deg_G(v) \le 3$, or $\deg_G(v)=  \frac{k}{2}$ for all $v\in V(G)$.\\
(d) Let $M$ be a maximum matching in $G$. Suppose $M$ has no edge incident with $x$. Then there is an edge $yz \in M$ where $\{y, z\} \cap N(x) \neq \emptyset$, otherwise we can take any $v \in N(x)$ and $M \cup \{xv\}$ would be a matching larger than $M$, a contradiction. So we may assume that $y \in N(x)$ and now we can take the maximum matching $M' = (M \setminus \{yz\}) \cup \{ xy \}$ which contains  an edge incident with $x$.
\end{proof}

We show in the next lemma that, for even $k$, the family $\mathcal{C}_k$ contains precisely those graphs on $k$ edges that have constant balancing number.

\begin{lemma}\label{lem:bal_const_odd_k}
Let $k \ge 2$ be an even integer, and let $G$ be a graph with $e(G)=k$. Then, there is a constant $c(G)>0$ such that, for every large enough $n$,  ${\rm bal}(n,G) = c(G)$  if and only if $G\in \mathcal{C}_k$.
\end{lemma}
\begin{proof}
Suppose first that ${\rm bal}(n,G) = c(G)$ for sufficiently large $n$. Let $n$ be an integer satisfying  that every $2$-edge-coloring of $K_n$ with $\min \{e(R), e(B)\} >c(G)$ contains a balanced copy of $G$, and such that $\left\lfloor\frac{n}{2} \right\rfloor >c(G)$. In particular, a coloring of $E(K_n) = E(R) \cup E(B)$ where the red edges induce a maximal matching of $K_n$ satisfies $\min\{e(R) , e(B)\} = \left\lfloor\frac{n}{2} \right\rfloor> c(G)$ and, thus, contains a balanced copy of $G$. Since a balanced copy of $G$ must have  exactly $\frac{k}{2}$ red edges, we conclude that $\beta (G)\geq \frac{k}{2}$. Consider now a coloring of $E(K_n) = E(R) \cup E(B)$, where the red edges induce a star $K_{1,n-1}$. This coloring satisfies $\min\{e(R) , e(B)\} =n-1> \left\lfloor\frac{n}{2} \right\rfloor > c(G)$ and, thus, contains a balanced copy of $G$. As all red edges are incident with a single  vertex, and all edges incident to that vertex are red,  we deduce that  $G$ must have a vertex $u$ with   $\deg_G(u)=\frac{k}{2}$. We conclude that $G\in \mathcal{C}_k$, as desired.

Suppose now that $G\in \mathcal{C}_k$. Since isolated vertices do not affect here, we can assume, without loss of generality, that $G$ has minimum degree at least $1$. In order to prove that there is a constant $c(G)$ such that ${\rm bal}(n,G) = c(G)$ for all sufficiently large $n$, we will use Theorem \ref{thm:constant} with $t=k+1$. Let $T=R(2k+2)$ and let $n$ be as large as required for Theorem \ref{thm:constant}. Consider a $2$-edge-coloring of $K_n$ with  $\min \{e(R), e(B)\} >(k+1)(2T^2-6T+5)$. By Theorem \ref{thm:constant}, we know that such a coloring contains either an induced monochromatic $K_{1,k+1}$, or an induced monochromatic $(k+1)K_2$. We will show that, in both cases, we can find a balanced copy of $G$. For the first case, suppose, without loss of generality, that there is an induced red $K_{1,k+1}$. Let $u$ be the central vertex of this star and consider the blue clique, say $U$, containing its $k+1$ neighbors. Observe that, in order to have a balanced copy of $G$, we can choose  appropriately $k$ edges the following way: just take $\frac{k}{2}$ red edges incident  to $u$ and choose the remaining  $\frac{k}{2}$ blue edges from $U$ as needed. In the second case, we can find a balanced copy of $G$ as follows. Let $W$ be the set of vertices of the induced monochromatic copy of $(k+1)K_2$, which we assume to be red. Recall that, by Lemma \ref{lem:structure}(b), there is an edge  $e\in E(G)$ such that $G-e$ can be decomposed into a star $K_{1, \frac{k}{2}}$ and the graph $(\frac{k}{2}-1)K_2$. We will find a balanced copy of $G$ where $\frac{k}{2}-1$ edges of the star plus the edge $e$ will be blue, while one edge from the star together with the remaining $\frac{k}{2}-1$ disjoint edges will be red.  For this purpose, take a set of $\frac{k}{2}$ disjoint red edges, and let $w$ be a vertex incident to one of them. This vertex $w$ will play the role of the vertex in $G$ with degree $\frac{k}{2}$, while the $\frac{k}{2}$ disjoint red edges will constitute a matching of $G$ (possibly maximum). Note that this is possible because of item (d) of Lemma \ref{lem:structure}. Even though the vertices from each of the $\frac{k}{2}-1$ disjoint edges can intersect the set of leaves from the star in $0$, $1$ or $2$ elements, we have enough room in $B[W]$ to choose the $\frac{k}{2}-1$ blue edges incident with $u$ plus the edge $e$.
\end{proof}

\begin{theorem}\label{thm:characterization}
Let $k \ge 2$ be an even integer, and let $G$ be a graph with $e(G)=k$. Then $G$ has constant balancing number if and only if $G\in \mathcal{C}_k$, if $k$ is even, and $G - e \in \mathcal{C}_{k-1}$ for some $e \in E(G)$, if $k$ is odd.
\end{theorem}
\begin{proof}
The case that $k$ is even is already covered by Lemma \ref{lem:bal_const_odd_k}. So we let $k$ be odd. Suppose first that $G$ is a graph on $k$ edges having an edge $e \in E(G)$ such that $G - e \in \mathcal{C}_{k-1}$. Then there is a constant $C$ such that, for $n$ large enough, there is a balanced copy of $G - e$ in every coloring of $K_n$ with more than $C$ edges in each color. No matter what color is the edge which corresponds to $e$, we see easily that there is a balanced copy of $G$, implying that $\bal(n, G) \le C$.

To prove the converse, let $G$ be a graph on $k$ edges having constant balancing number, say $\bal(n,G) = C$, for $n$ sufficiently large. Let $n$ be an integer satisfying  that every $2$-edge-coloring of $K_n$ with more than $C$ edges in each color contains a balanced copy of $G$, and such that $\left\lfloor\frac{n}{2} \right\rfloor > C$. In particular, a coloring of $E(K_n) = E(R) \cup E(B)$ where the red edges induce a maximal matching of $K_n$ satisfies $\min\{e(R) , e(B)\} = \left\lfloor\frac{n}{2} \right\rfloor> C$ and, thus, contains a balanced copy of $G$. Since a balanced copy of $G$ has either $\frac{k-1}{2}$ or $\frac{k+1}{2}$ red edges, we conclude that $\beta (G)\geq \lfloor \frac{k-1}{2} \rfloor$. Consider now a coloring of $E(K_n) = E(R) \cup E(B)$, where the red edges induce a star $K_{1,n-1}$. This coloring satisfies $\min\{e(R) , e(B)\} = n-1 > \left\lfloor\frac{n}{2} \right\rfloor > C$ and, thus, contains a balanced copy of $G$. As all red edges are incident with a single  vertex, and all edges incident to that vertex are red,  we deduce that  $G$ must have a vertex $u$ with   $\frac{k-1}{2} \deg_G(u) \le \frac{k+1}{2}$. Now we distinguish two cases.

\noindent
\emph{Case 1. Suppose that $\deg_G(u) = \frac{k+1}{2}$.} Let $M$ be a maximum matching in $G$. By what we deduced above, we have $|M| \ge \frac{k-1}{2}$.  Since $\deg_G(u) = \frac{k+1}{2} \ge 2$, there exists an edge $e$ that is incident to $u$ but not part of $M$. Then $M$ is a matching in $G-e$ and thus $\beta(G-e) \ge |M| \ge \frac{k-1}{2}$. Since, moreover, $e(G-e) = k-1$ and $\deg_{G-e}(u) = \frac{k-1}{2}$, $G - e \in \mathcal{C}_{k-1}$, and we are done.

\noindent
\emph{Case 2. Suppose that $\deg_G(u) = \frac{k-1}{2}$.} Since, in any matching of $G$, there is at most one edge incident with $u$, we deduce, together with the lower bound obtained above, that $\frac{k-1}{2} \le \beta(G) \le \frac{k+3}{2}$. If $\beta(G) \ge \frac{k+1}{2}$, we consider an edge $e$ that is not incident with $u$, and notice that $\beta(G-e) \ge \beta(G)-1 \ge \frac{k-1}{2}$. On the other hand, if $\beta(G) = \frac{k-1}{2}$, then we take a maximum matching $M$, and note that there is at least one edge $e \in E(G) \setminus M$ that is not incident with $u$. Then $M$ is still a matching in $G-e$, implying that $\beta(G-e) \ge \frac{k-1}{2}$. Hence, in both cases we obtain that $G-e$ has $e(G-e) = k-1$, $\deg_{G-e}(u) = \frac{k-1}{2}$, and $\beta(G-e) \ge \frac{k-1}{2}$, yielding that $G-e \in \mathcal{C}_{k-1}$.
\end{proof}
With a much more detailed analysis via an inductive proof, one can get that $\bal(n,G) \le \frac{15}{8} k^2$ for any $G \in \mathcal{C}_k$ \cite{CGHJMM}.

\section{Forbidding weakly induced graphs}\label{sec:effect}
In Theorems \ref{thm:linear-tK_2} and \ref{thm:constant} we can see the effect of forbidding an induced matching of certain size in each of the colors. In this section, we explore the possibility of forbidding other induced monochromatic substructures.

We say that a graph $H$ is \emph{weakly induced} as a subgraph of a graph $G$, if $H$ appears as a subgraph of $G$ but there are no induced edges between the components of $H$ (in other words, the components of $H$ are pairwise at distance at least $2$ in $G$). Note that the notion of weakly induced subgraph makes only sense when $H$ is a non-connected graph, since otherwise it is simply a subgraph.

The following theorem states that,  for any connected bipartite graph $H$ of order at least $3$,  there is a constant $C= C(H,r,s,t)$ such that every $2$-coloring of $K_n$ with at least $C \ex(n,H)$ edges in each color contains either an induced monochromatic $K_{s,t}$, an induced monochromatic $S_{t,t}$, or a monochromatic weakly induced $rH$. The case where $H=K_2$ was considered in Theorem \ref{thm:linear-tK_2}. Observe that Theorem \ref{thm:forbidding_tH} does not generalize the latter because we need here a different hypothesis on the number of edges in each color. However, the proof works similarly.
\begin{theorem}\label{thm:forbidding_tH}
Let $r,s,t$ be positive integers such that $r \ge 2$, and $t \ge \max\{2,  s\}$, and let $H$ be any connected bipartite graph of order at least $3$. Then there is a constant $C= C(H,r,s,t)$ such that, if $n$ is large enough, then every coloring $E(K_n) = E(R) \cup E(B)$ with $\min \{ e(R), e(B)\} \ge C \ex(n, H)$, but without a monochromatic weakly induced $rH$, contains an induced monochromatic member from $\mathcal{F}_{s,t}$.
\end{theorem}
Before proceeding to prove Theorem \ref{thm:forbidding_tH}, we need a result by Gorgol \cite{Gor11}, which states that, for any connected graph $H$ and any integer $q \ge 1$,
\begin{equation}\label{eq:Gorgol}
\ex(n,qH) = \ex(n,H) + \mathcal{O}(n).
\end{equation}
note that, if $H$ is a connected graph with $\ex(n,H) = \Omega(n)$, then there is a constant $C$ such that $C \ex(n,H) \ge \ex(n,qH)$. But the condition $\ex(n,H) = \Omega(n)$ is equivalent to saying that $H$ has order at least $3$ because the only connected graph $H$ with $\ex(n,H) = o(n)$ is precisely $K_2$ (it is the only connected non-trivial subgraph contained in both $K_{1,n-1}$ and $\lfloor \frac{n}{2}\rfloor K_2$, which have both a linear number of edges, and clearly $\ex(n,K_2) = 0$).

As in Section \ref{sec:evolution}, we will make use of the Zarankievicz number ${\rm z}(n;t)$ defined as the maximum number of edges in a $K_{t,t}$-free bipartite graph with $n$ vertices in each part.

\begin{proof}[Proof of Theorem \ref{thm:forbidding_tH}]
Let $m = n(H)$,  and let $\ell, q$ be integers with $\ell^2 \ge \max\{ 2 z(\ell,t), 4r \binom{m}{2} z(\ell,\ell;s,t) \}$, $\ell \ge t$, and $q \ge Q(\ell,m)$.  Let $n \ge R(\ell) + 2qm$ and let $C$ be large enough such that $C \ex(n,H) \ge \ex(n,qH)$, which exists by what we discussed before starting the proof.  Suppose we have a complete graph $K_n$ provided with a red-blue coloring $E(K_n) = E(R) \cup E(B)$ such that $\min \{ e(R), e(B)\} \ge C \ex(n,H)$ and such that it has no monochromatic weakly induced $rH$ (this can happen for $r \ge 2$). Observe that the condition on the amount of edges in each color class is possible because $H$ is bipartite, and thus $\ex(n,H) = o(n^2)$. Then there is both, a red and a blue copy of $qH$.  Consider the set $S$ of vertices that do not belong to any of these copies.  Since $|S| \ge R(\ell)$, there is a monochromatic $\ell$-clique $S'$ disjoint from the red and the blue copies of $qH$. Without loss of generality, we can assume that $S'$ is a red clique. Now consider the blue copy of $qH$ and let $H_1, H_2 \ldots, H_q$ the different copies of $H$. Label with $v_{i,j}$, $i \in [m]$ the vertices from $H_j$, for $j \in [q]$.  Since $q \ge Q(\ell,m)$, there  is a set $L \subseteq [q]$ with $|L| = \ell$ such that, for each $i \in [m]$, the set $S_i = \{v_{i,j} \;|\; j \in L\}$ induces a monochromatic $K_{\ell}$.  If there is some $i \in [m]$ such that $S_i$ induces a blue $K_{\ell}$, then we can consider the most represented color in $E(S',S_i)$, which has at least $\frac{{\ell}^2}{2} > z(\ell;t)$ edges.  Hence, it can be easily seen that  there is a monochromatic $K_{t,t}$ with one partition set inside $S$ and the other inside $S_i$. In either case, if this $K_{t,t}$ is red or blue, we obtain an induced monochromatic $S_{t,t}$. Thus, we may assume that all cliques $S_i$, $i \in [m]$, are red.
Define a graph $F$ with vertex set $V(F) = \{w_j \;|\; j \in L \}$ and edge set $E(F) = \{ w_{j_1}w_{j_2} \; |\; e_B(H_{j_1}, H_{j_2}) \neq 0\}$.  If $\alpha(F) \ge r$,  this  would imply,  by the construction of $F$, that there is a blue weakly induced $rH$, which is not possible. Hence,  we can assume that $\alpha(F) < r$, and by Tur\'an's bound on the independence number \cite{Tur} we have
\[r > \alpha(F) \ge \frac{\ell}{d(F)+1} = \frac{{\ell}^2}{2e(F)+\ell},\]
implying that $e(F) > \frac{{\ell}^2}{2t} - \frac{\ell}{2} \ge \frac{{\ell}^2}{4r}$.  By construction of the graph $F$, every edge $e \in E(F)$ corresponds to at least one blue edge $e' \in E(H_{i_1}, H_{i_2})$ for some pair $i_1, i_2 \in L$.  Moreover, since among the sets $S_i$, $i \in [m]$, there are only red edges, we can derive
\[\frac{{\ell}^2}{4r} < e(F) \le \sum_{j_1, j_2 \in L} e_B(H_{j_1}, H_{j_2}) = \sum_{i_1, i_2 \in [m], i_1 \neq i_2} e_B(S_{i_1}, S_{i_2}).\]
As there are $\binom{m}{2}$ different pairs $(i_1, i_2) \in [m]^2$, there is at least one pair $(i_1^*,i_2^*) \in [m]^2$ such that
\[e_B(S_{i_1^*},S_{i_2^*}) > \frac{{\ell}^2}{4r \binom{m}{2}} \ge z(\ell,\ell;s,t),\]
implying that there is a blue $K_{s,t}$ with one partition set in $S_{i_1^*}$ and the other in $S_{i_2^*}$. Hence, there is an induced blue $K_{s,t}$.
\end{proof}
Observe that, in Theorem \ref{thm:forbidding_tH}, the condition that $H$ is bipartite is best possible in the sense that any non-bipartite graph $H$ has $\ex(n,H) \ge \frac{n^2}{4}$ (see \cite{FuSi-survey}) and, because of this, we cannot have a coloring  $E(K_n) = E(R) \cup E(B)$ such that $\min \{ e(R), e(B)\} \ge C \ex(n,H)$, as needed in the proof. In contrast, if $T$ is a tree of order at least $3$, we know that $\ex(n,T)$ is linear in $n$ (see \cite{FuSi-survey}). As a consequence, we deduce with (\ref{eq:Gorgol}) that  
\begin{equation}\label{eq:GorgolT}
\ex(n, qT) = \Theta(n)
\end{equation}
 for any tree $T$ of order at least $3$ and any integer $q \ge 1$.  Hence,  Theorem \ref{thm:forbidding_tH} implies that a sufficiently large linear constraint on the color classes is enough to guarantee the existence of either a monochromatic weakly induced $rT$, an induced monochromatic $K_{s,t}$,  or an induced monochromatic $S_{t,t}$. 
\begin{corollary}\label{cor:linear-weak-ind-tT}
Let $r,s,t$ be positive integers  such that $r \ge 2$, $t \ge \max\{2,  s\}$. Then, for any tree $T$  of order at least $3$, there is a constant $C= C(T,r,s,t)$  such that, if $n$ is large enough,  then  every coloring $E(K_n) = E(R) \cup E(B)$ with $\min \{ e(R), e(B)\} \ge C n$ without a monochromatic weakly induced $rT$, contains an induced monochromatic member from $\mathcal{F}_{s,t}$.
\end{corollary}
In this case we obtain again sharpness by means of the example given in Figure \ref{fig:lineal_sharp}.\\

\section{Multicolor setting}\label{sec:multic}

One possible and natural line that could be pursued within this topic is to consider more colors. In this sense, for an integer $k \ge 2$, given a family $\mathcal{F}$ of $k$-colored graphs, $\ex_k(K_n,\mathcal{F})$ would be the maximum integer $m$ (if it exists) such that there is a $k$-coloring of the edges $K_n$ with $m$ edges in the smallest color class but without a color-matched copy of any of the members of $\mathcal{F}$. In \cite{BHMM}, a multicolor version of (\ref{eq:CHM}) is given. More precisely, it is shown that there is a family $\mathcal{F}^k_t$ of $k$-colored complete graphs that is unavoidable if the coloring has at least $C n^{2-\frac{1}{tk^k}}$ edges in each color, where $C= C(t,k) >0$ is a constant. It is conjectured that $n^{2-\frac{1}{t}}$ is the correct order of magnitude \cite{BHMM}. The family $\mathcal{F}^k_t$ consists of $k$-colored complete graphs whose vertices can be partitioned into monochromatic complete graphs of order $t$ such that the complete bipartite graph between any two of them is monochromatic, too. Moreover, all $k$ colors are present in every member of $\mathcal{F}^k_t$ and no member is properly contained in another. It turns out here too that, assuming there is at least certain constant amount of edges in each color, monochromatic stars or monochromatic matchings are unavoidable. Moreover, if there are colors where the latter does not happen, color patterns similar to $\mathcal{F}^k_t$ emerge as well. \\

\begin{theorem}\label{thm:multicol}
Let $k ,r, t  \ge 2$ be integers. For $n$ large enough, there is a constant $c(k,r,t)$ such that, if  $f: E(K_n) \to [k]$ is a $k$-coloring of the edges of $K_n$ and $\min \{ |f^{-1}(i)| \;|\; i \in [k]\} \ge c(k,r,t)$, $A_f$ is the set of all $i \in [k]$ such that the graph induced by the edge set $f^{-1}(i)$ contains an induced monochromatic member from $\{ K_{1,t} , tK_2\}$, and $B_f = [k] \setminus  A_f$, then the following hold:  
 \begin{enumerate}
     \item[(a)] $A_f \neq \emptyset$;
     \item[(b)] For any $i\in B_f$, the graph induced by  $f^{-1}(i)$ contains a clique $C_i$ of size $t$, such that, for any $j \in B_f$, $i \neq j$, the set $C_i \cup C_j$ induces a monochromatic $K_{t,t}$ of some color from $A_f$.
 \end{enumerate}
\end{theorem}

\begin{proof}
Let $n$ be large enough such that the following argumentation remains valid. Let $b = \binom{|B_f|}{2}$, and $t_1 =  BR(k;t)$, where $BR(k;t)$ is the bipartite Ramsey number for $k$ colors and $K_{t,t}$. We define recursively $t_a = BR(k;t_{a-1})$, for $2 \le a \le b$, and we set $q = t_b + b$. Let $c(t) = C(q,q)$ be the constant of Theorem \ref{thm:constant}, and suppose that $f: E(K_n) \to [k]$ is a $k$-coloring of the edges of $K_n$ such that $\min \{ |f^{-1}(i)| \;|\; i \in [k]\} \ge c(t)$.  Let $G_i$ be  the graph induced by the edge set $f^{-1}(i)$.

We will show first that $G_i$ contains either an induced $K_{1,q}$, an induced $qK_2$ or a $K_q$, for every $i \in [k]$. Consider a coloring $E(K_n) = E(R) \cup E(B)$ such that $R \cong G_i$. By Theorem \ref{thm:constant}, there is a monochromatic induced $qK_2$ or a monochromatic induced $K_{1,q}$. If any of these is blue, we see easily that we have a red $K_q$. Otherwise, there is a red induced member from $\{ K_{1,q} , qK_2\}$.

Clearly, $A_f \neq \emptyset$ if $|B_f| \le 1$. Now suppose $|B_f| \ge 2$. By what we showed above, this means that $G_i$ contains a clique $C'_i$ of size $q = t_b + b$, for $1 \le i \le b$. Observe that any two such cliques can have at most one vertex in common as they are of different colors. Since $q = t_b + b$, and $b = \binom{|B_f|}{2}$, we can select $t_b$ vertices from each clique $C'_i$ to obtain pairwise disjoint cliques $C_i$, $1 \le i \le b$, of size $t_b$ each. Take now two indices $i, j \in B_f$, $i \neq j$.
Since $t_b =  BR(k;t_{b-1})$, there has to be a monochromatic $K_{t_{b-1},t_{b-1}}$ between $C_i$ and $C_j$. Say this $K_{t_{b-1},t_{b-1}}$ has color $\ell$, which can be equal to $i$ or $j$, but clearly not to both. In any case, this complete bipartite graph contains an induced $K_{1,t}$ as a subgraph. In particular, we have shown that  $\ell\in A_f$, thus $A_f \neq \emptyset$, completing the proof for (a). To finish the proof of (b), one can proceed recursively as above for all pairs of cliques. In each step, we select a subset from the cliques fulfilling that the edges between them are monochromatic. This can be done by selecting, at step $s$, $t_{b-s} = BR(k;t_{b-s-1})$ vertices from each clique. Note that edges between the pairs already covered are all of the same color, so selecting subsets from the cliques does not represent a problem between these pairs. For the new pair however, we need to carefully select $t_{b-s}$ vertices such that all edges in between have the same color. This can be done because $t_{b-s+1} = BR(k;t_{b-s})$. In this way, the cliques become smaller in each step, starting from size $t_b$, $t_{b-1}$, and so on until size $t_1 = BR(k;t) $, and finally, in the last step, all the cliques are of size $t$. Since the edges among the cliques contain induced $K_{1,t}$'s, it follows that all these edges have colors from $A_f$.
\end{proof}

Observe that we cannot do much better in Theorem \ref{thm:multicol}. Consider $k\ge 3$, and divide the vertex set of $K_n$ into $k-1$ disjoint sets as equal in size as possible, making each of them monochromatic with one color from $\{1, \ldots, k-1\}$. The remaining edges are colored with color $k$. All colors $1, \ldots, k-1$ have only monochromatic complete graphs and no induced star or induced matching.  Only color $k$ has not only induced stars but an induced complete $(k-1)$-partite graph with equal parts.\\

\section*{Acknowledgements}

The second and third authors were partially supported by PAPIIT IG100822 and CONACyT project 282280.


\begin{thebibliography}{100}

\bibitem{ABHS} H. L. Abbott, D. Hanson, and N. Sauer, Intersection theorems for systems of sets, J. Comb. Th. Ser. A 12 (1972), 381–389.

\bibitem{AiZie} M. Aigner, G. M. Ziegler, Chapter 41: Turán's graph theorem, Proofs from THE BOOK (6th ed.), Springer-Verlag (2018), 285--289.
\bibitem{ARS} N. Alon, L. Rónyai, and T. Szabó, Norm graphs: variations and applications. J. Combin. Theory Ser. B {\bf 76} (1999), no. 2, 280--290.
\bibitem{AlSh} N. Alon, C. Shikhelman,  Many T copies in $H$-free graphs, J. Combin. Theory, Ser. B, {\bf 121} (2016), 146--172.
\bibitem{AxGo} M. Axenovich, I. Gorgol, Induced Ramsey number for a star versus a fixed graph, Electron. J. Combin. {\bf 28} (2021), no. 1, P. 1.55, 13 pp. 
\bibitem{BHMM} M. Bowen, A. Hansberg, A. Montejano, A. M\"uyesser, Colored unavoidable patterns and balanceable graphs, arxiv:1912.06302

\bibitem{BLM} M. Bowen, A. Lamaison, Alp M\"uyesser, Finding unavoidable colorful patterns in multicolored graphs, Electron. J. Combin. {\bf 27} (2020), no. 4, P. 4.4.

\bibitem{Bro66} W. G. Brown, On graphs that do not contain a Thomsen graph, Canad. Math. Bull. {\bf 9} (1966), 281--285.

\bibitem{BrJo18} H. Bruhn, F.  Joos, A stronger bound for the strong chromatic index, Combin. Probab. Comput. {\bf 27} (2018), no. 1, 21--43. 

\bibitem{CGMS23} M. Campos, S. Griffiths, R. Morris, J. Sahasrabudhe, An exponential improvement for diagonal Ramsey, arXiv:2303.09521 (2023). 

\bibitem{CGHJMM} Y. Caro, I. González-Escalante, A. Hansberg, M. Jácome, T. Matos, and A. Montejano, Graphs with constant balancing number, submitted.

\bibitem{CHM_amo} Y. Caro, A. Hansberg, and A. Montejano, Graphs isomorphisms under edge-replacements and the family of amoebas,  Electron. J. Combin., accepted.

\bibitem{CHM_unav} Y. Caro, A. Hansberg, and A. Montejano, Unavoidable chromatic patterns in 2-colorings of the complete graph, J. Graph Theory {\bf 97} (2021), 123--147.

\bibitem{CHM-K_4} Y. Caro, A.  Hansberg, A.  Montejano. Zero-sum $K_m$ over $\mathbb{Z}$ and the story of $K_4$, Graphs Combin. 35 (2019), no. 4, 855--865.

\bibitem{CLZ19} Y. Caro, J. Lauri, C. Zarb, On small balanceable, strongly-balanceable and omnitonal graphs, Discussiones Mathematicae Graph Theory {\bf 42} (2022), 1219--1235.

\bibitem{CFKP20} I. Choi, M. Furuya, R. Kim, B. Park, A Ramsey-type theorem for the matching number regarding connected graphs, Discrete Math. {\bf 343} (2) 111648 (2020), 6 pp. 

\bibitem{Chud} M. Chudnovsky, "The Erd\H{o}s–Hajnal conjecture -- a survey", J. Graph Th. {\bf 75} (2014), no. 2, 178--190.
 
  \bibitem{cutler}
  J. Cutler and B. Mont\'agh, Unavoidable subgraphs of colored graphs, Discrete Math. {\bf 308} (2008), 4396--4413.
 
\bibitem{DHV20} A. Dailly, A. Hansberg, A., D. Ventura, On the balanceability of some graph classes, Discrete Applied Mathematics {\bf 291} (2021), 51--63.

\bibitem{DHV_Kn} A. Dailly, A. Hansberg, A. Talon, D. Ventura,  Balanceability of complete graphs, preprint.
  
\bibitem{ErHa} P. Erd\H{o}s, A. Hajnal, A., Ramsey-type theorems, Discrete Applied Mathematics, {\bf 25} (1989) no. 1--2, 37--52.
\bibitem{ErRa} P. Erd\H{o}s, R. Rado, Intersection theorems for systems of sets, J. London Math. Soc., {\bf 35} (1), 85--90.
\bibitem{ERS66} P. Erd\H{o}s, R. R\'{e}nyi and V. T. S\'{o}s, On a problem of graph theory, Studia Sci. Math. Hungar. {\bf 1} (1966), 215--235.


\bibitem{ErSz} P.  Erd\H{o}s, G. Szekeres, A combinatorial problem in geometry, Compositio Mathematica {\bf 2} (1935), 463--470.
   
\bibitem{EHMV_Fibo} L. Eslava, A. Hansberg, T. Matos, D. Ventura, New recursive constructions of amoebas and their balancing number, preprint.

\bibitem{FGST89} R. J. Faudree, A. Gy\'arf\'as, R. H. Schelp, and  Z. Tuza, Induced matchings in bipartite graphs, Discrete Math. {\bf 78} (1989),  83--87.


\bibitem{FoSu_DRC} J. Fox, B. Sudakov, Dependent random choice, Random Structures Algorithms {\bf 38} 1--2 (2011), 68--99.

\bibitem{FoSu_indRam} J. Fox, B. Sudakov, Induced Ramsey-type theorems, Adv. Math. {\bf 219} (2008), no. 6, 1771--1800.


\bibitem{FoSu}
  J. Fox, B. Sudakov, Unavoidable patterns, J. Combin. Theory Ser. A {\bf 115} (2008), no. 8, 1561--1569.
  
\bibitem{Fur96} Z. F\"{u}redi, On the number of edges of quadrilateral-free graphs, J. Comb. Theory Ser. B {\bf 68} (1996), 1--6.

\bibitem{FuSi-survey} Z. F\"uredi, M. Simonovits, The history of degenerate (bipartite) extremal graph problems, Erd\H{o}s centennial, Bolyai Soc. Math. Stud., 25, János Bolyai Math. Soc., Budapest (2013), 169--264. 

\bibitem{GiHa} A. Gir\~{a}o, R. Hancock, Two Ramsey problems in blowups of graphs, arXiv:2205.12826.

\bibitem{GiMu} A. Gir\~{a}o, D. Munh\'{a} Correia, A note on unavoidable patterns in locally dense colourings, arXiv:2211.01862.


\bibitem{GiNa} A. Gir\~{a}o, B. Narayanan, Tur\'{a}n theorems for unavoidable patterns, Math. Proc. Cambr. Phil. Soc. (2021), 1--20.

\bibitem{Gor11} I. Gorgol, Tur\'an Numbers for Disjoint Copies of Graphs, Graphs Combin.  {\bf 27} (2011), 661--667.


\bibitem{GrSp} R. Graham, B. Rothschild, J. H. Spencer, (1990), Ramsey Theory, New York: John Wiley and Sons.

\bibitem{Ill} F. Illingworth, A note on induced Turán numbers (2021), arXiv:2105.12503.

\bibitem{KaMu} N. Kam\v{c}ev, A. Müyesser, Unavoidable patterns in locally balanced colourings, arXiv:2209.06807.

\bibitem{KRS} J. Kollár, L. Rónyai, T. Szabó, Norm-graphs and bipartite turán numbers. Combinatorica 16, 399–406 (1996).

\bibitem{KST} T. K\H{o}v\'{a}ri, V. T. S\'{o}s, P.  Tur\'{a}n, On a problem of K. Zarankiewicz, Colloq. Math. {\bf 3} (1954), 50--57.

\bibitem{LTTZ} P. Loh, M. Tait, C. Timmons, R. Zhou, Induced Turán Numbers. Combinatorics, Probability and Computing {\bf 27}(2)  (2018), 274--288.

\bibitem{MoRe97} M. Molloy, B. Reed, (1997), A bound on the strong chromatic index of a graph. J. Combin. Theory Ser. B {\bf 69} 103--109.

\bibitem{MuTa} A. M\"uyesser, M. Tait, Turán and Ramsey-type results for unavoidable subgraphs, J. Graph Theory {\bf 101} (2022), no. 4, 597--622.

\bibitem{Ram} F. P. Ramsey,  On a problem of formal logic, Proceedings of the London Mathematical Society, 30 (1930), 264--286.

\bibitem{Tur} P. Turán, On an extremal problem in graph theory, Matematikai és Fizikai Lapok (in Hungarian) {\bf 48} (1941), 436--452.

\bibitem{Ven23} D. Ventura, Unavoidable patterns, balanceability and amoebas, PhD. thesis, Universidad Nacional Autónoma de México, 2023.

\bibitem{Zar51} K. Zarankiewicz, Problem P 101, Colloq. Math. {\bf 2} (1951), 116--131.
\end{thebibliography}
	\end{document}